\documentclass[10pt, oneside]{article}
\usepackage{geometry}                		
\usepackage[pagewise]{lineno}	
\usepackage[parfill]{parskip}    		
\usepackage{graphicx}				
\usepackage[utf8]{inputenc}
\usepackage[english]{babel}
\usepackage{mathrsfs,amsmath}
\usepackage{amsthm}								
\usepackage{amssymb}
\usepackage{amsmath}
\usepackage{amsfonts}
\usepackage{mathbbol}
\usepackage{biblatex}
\usepackage{xpatch}
\usepackage{subcaption}
\usepackage{hyperref}
\usepackage{csquotes}
\usepackage{lipsum}

\newtheorem{theorem}{Theorem}[section]
\makeatletter
\AtBeginDocument{\xpatchcmd{\@thm}{\thm@headpunct{.}}{\thm@headpunct{}}{}{}}
\makeatother
\newtheorem{lemma}[theorem]{Lemma}

\newtheorem{proposition}[theorem]{Proposition}

\newtheorem{corollary}[theorem]{Corollary}

\theoremstyle{definition}
\newtheorem{definition}{Definition}[section]

\theoremstyle{remark}

\author{Wei Lin}
\date{}
\title{A MERIDIAN LEMMA FOR FULLY ALTERNATING LINKS IN THICKENED SURFACES}

\begin{document}
\maketitle

\begin{abstract}

Menasco showed that a closed surface in the complement of a non-split prime alternating link in $S^3$ contains a circle isotopic in the link complement to a meridian of the links. This result is known as the meridian lemma for alternating links. We give a meridian lemma for the class of fully alternating links in the thickened orientable surfaces. 

\end{abstract}

\section{Introduction}
The class of alternating knots and links is a beloved subject in knot theory. There are various generalizations of alternating links. The definition of fully alternating links in thickened surfaces is given in \hyperref[B3]{[3]} by Adams et al.:

\begin{definition}
Let $L$ be a link in a thickened surface $S\times I$, orientable or not, with the exception of the sphere and the projective plane. A projection of $L$ from $S\times I$ to $S$ is \emph{fully alternating} if it is alternating on $S$ and the interior of the closure of every complementary region is an open disk. We say a link $L$ is fully alternating in $S\times I$ if it has a fully alternating projection from $S\times I$ to $S$.\par

\end{definition}

Note that the assumption of alternatingness implies that in the case $S$ is orientable, each component must have at least two crossings. In the case $S$ is non-orientable, there could be a component with Möbius band neighborhood with one crossing. Throughout this paper, we use the same notation of \hyperref[B3]{[3]}. And we say a surface properly embedded in a compact manifold is \emph{essential} if it is incompressible and not boundary-parallel.\par

\subsection{Main Result}

Let $S$ be a closed orientable surface with positive genus, $L$ be a fully alternating link in $S\times I$. Adams et al. \hyperref[B3]{[3]} have shown that if $F\subset S\times I \setminus L$ is an essential torus, then $F$ contains a circle which is isotopic in $S\times I\setminus L$ to a meridian of $L$. Expanding Menasco's idea \hyperref[B16]{[16]}, we can simplify some parts of the proofs given in \hyperref[B3]{[3]}, and show a result that applies to arbitrary closed essential surfaces embedded in the same manifold structure: \par

\begin{theorem}\label{mfa}
Let $S$ be a closed orientable surface, $L$ be a fully alternating link in $S\times I$. If $F \subset S\times I\setminus L$ is a closed essential surface, then $F$ contains a circle which is isotopic in $S\times I\setminus L$ to a meridian of $L$.\par
\end{theorem}

For reader familiar with the arguments in \hyperref[B16]{[16]}, our approach to proving Theorem \ref{mfa} should be of interest. It also provides a similar but different way to prove the meridian lemma for alternating links \hyperref[B16]{[16]}. To be more specific, let $M=S\times I$, we cut $M$ and consider one of the submanifolds, say $M_+$. Then we consider the essential surface $F\subset M\setminus L$, and the subsurface components of $F \cap M_+$, which are called domes (please see \S  \ref{3} for formal definition). And we find that the ``nesting behavior'' will be imposed by domes instead of by simple closed curves in the arguments of \hyperref[B16]{[16]}. The assumption of an fully alternating projection then forces the manifesting of a meridian from the lowest dome.\par

In addition, the authors \hyperref[B3]{[3]} give a definition for prime links in thickened surfaces in comparison with the prime links in $S^3$: \par

\begin{definition}
A link $L$ is \emph{prime} in $S\times I$ if there does not exist an essential twice punctured sphere in $S\times I\setminus L$ such that both punctures are created by $L$. \par
\end{definition}

We place $F$ into \emph{normal position} (please see \S  \ref{2} for formal definition) with respect to $S_\pm$. One of the main goals for the application of such techniques in \hyperref[B16]{[16]}, or in some other generalizations (please see \hyperref[B1]{[1]}, \hyperref[B2]{[2]}, \hyperref[B3]{[3]}, and \hyperref[B4]{[4]}), is when one imposes the assumption that the projection $\pi(L)$ on the projection surface $S$ is ``alternating'', the normal position of an essential surface is exceedingly well-behaved to the point where by direct observation one can definitively state whether the link is split, prime, cabled or a satellite. In \hyperref[B3]{[3]}, the authors have shown that when $L$ is prime, there does not exist any essential annuli in $S\times I \setminus L$. Then since there are also no essential spheres or tori in $S\times I \setminus L$, they can give the following result by Thurston's Hyperbolization Theorem \hyperref[B21]{[21]}: \par

\begin{theorem}
Let $S$ be a closed orientable surface. A prime fully alternating link $L$ is $S \times I\setminus L$ is hyperbolic.\par

\end{theorem}
\par

\subsection{Background}

Let $ L \subset S^3$ be a link and $\pi(L) \subset S^2 (\subset S^3)$ be its projection on $S^2$. In 1981 Menasco \hyperref[B16]{[16]} introduced his crossing ball technology for link projections  that replaced $\pi(L)$ in $S^2$ with two 2-spheres, $S^2_\pm$, which had the salient features that $L$ was embedded in $S^2_+ \cup S^2_-$ and $S^3 \setminus (S^2_+ \cup S^2_-)$ was a collection of open $3$-balls---$B^3_\pm$ that correspond to the boundaries $S^2_\pm $ and a collection of {\em crossing balls}.  (Please see \S  \ref{3} for formal definition.)  The techniques \hyperref[B16]{[16]} provide a way for arguing the existence of a meridional curve. In brief, once a surface $F$ is in ``normal position'', one can consider how the simple closed curves (s.c.c.'s) of $F\cap S^2_{\pm}$ impose a ``nesting behavior'' on a subset of s.c.c.'s in $F \cap S^2_\mp$, the assumption of an alternating projection forces the existence of some innermost s.c.c. from this subset manifesting a meridian.

One salient result from \hyperref[B16]{[16]} is that any essential surface in a non-split alternating link exterior will contain a meridional curve of a link component and, thus, studying such essential surfaces can be reduced to studying essential surfaces with meridional boundary curves that are \emph{meridionally incompressible}, or equally in some literature, pairwise incompressible.  The importance of studying meridionally incompressible surfaces has been reflected in the work of numerous scholars. To name a few, Bonahon and Seibenmann's work on arborescent knots \hyperref[B6]{[6]}, Oertel's work on star links \hyperref[B17]{[17]}, Adams' work on toroidally alternating links \hyperref[B2]{[2]}, Adams' et al. work on almost alternating links \hyperref[B4]{[4]}, Fa's initial cataloging of incompressible meridionally incompressible patterns \hyperref[B7]{[7]}, Lozanoand-Przytycki work on $3$-braid links \hyperref[B14]{[14]}, and Hass-Thompson-Tsvietkova results on growth of the number of essential surfaces in alternating link complements \hyperref[B9]{[9]}.

The definition of alternating link is described through diagrams, but characterizations without involving diagrams was found by Greene \hyperref[B8]{[8]} and Howie \hyperref[B10]{[10]} independently.  Howie and Purcell \hyperref[B12]{[12]} have proved various facts about the hyperbolic geometry of alternating links on surfaces, including weakly generalized alternating links described by Howie \hyperref[B11]{[11]}.

Some of these generalizations originated out of attempts to extend the results known for alternating links to broader classes of links.  Here is a list of some generalizations of alternating knots and links.\par

\begin{itemize}

\item almost alternating, m-almost alternating (Adams et al \hyperref[B4]{[4]})
\item toroidally alternating           (Adams \hyperref[B2]{[2]})
\item fully alternating                (Adams et al \hyperref[B3]{[3]})
\item adequate                         ((Lickorish‐Thistlethwaite \hyperref[B13]{[13]})
\item semi-alternating                  ((Lickorish‐Thistlethwaite \hyperref[B13]{[13]})
\item pseudo-alternating                     (Mayland‐Murasugi \hyperref[B15]{[15]})
\item algebraically alternating               (Ozawa \hyperref[B18]{[18]})
\item biperiodic alternating                   (Adams, Calderon, and Mayer \hyperref[B5]{[5]})
\item quasi-alternating                         (Ozsváth‐Szabó \hyperref[B19]{[19]})\\ \\
\end{itemize}

In \hyperref[B16]{[16]} Menasco showed the following theorem, known as the meridian lemma for alternating links, listed as below:\par

\begin{theorem}\label{meridian lemma}
If $L$ is a non-split prime alternating link, and if $F \subset S^{3}\setminus L$ is a closed incompressible surface, then $F$ contains a circle which is isotopic in $S^{3}\setminus L$ to a meridian of $L$.\par
\end{theorem}

The following definition was given in \hyperref[B16]{[16]} for the convienence of studying closed incompressible surfaces, or incompressible surfaces with meridional boundaries in the link complement. The reason for this definition is, we can cut a closed incompressible surface containing meridional curves along some non-parallel meridional curves, so that the surface becomes a union of punctured incompressible subsurfaces, each being incompressible and possessing a good property called \emph{meridionally incompressible}:\par

\begin{definition}\label{mi}
A surface $F$ in the complement of a link $L$ in a 3-manifold $M$ is called meridionally incompressible, if for each disk $D\subset M$ meeting $L$ transversely in one point, with $D\cap F=\partial D$, there is a disk $D'\subset F\cup L$ meeting $L$ transversely in one point, with $\partial D'= \partial D$. Otherwise, we call $F$ \emph{meridionally compressible}.
\end{definition}

The model and techniques in \hyperref[B16]{[16]} also applies to several other generalizations. In \hyperref[B4]{[4]} Adams et al. have shown a version of meridian lemma for the class of almost alternating knots. They showed that there is no closed incompressible meridionally incompressible surface in the complement of an almost alternating knot in $S^3$. Adams \hyperref[B2]{[2]} defined a broader class of knots that includes Montesinos knots and almost alternating knots, called toroidally alternating knots, which are often considered embedded in the lens space $L(p,q)$. It turns out when $p$ is odd, the meridian lemma also applies to toroidally alternating knots in $L(p,q)$, as a result of the following theorem:\par

\begin{theorem}

Let $K$ be a toroidally alternating knot in a genus one manifold $M_L=L(p,q)$. If there exists a closed orientable incompressible meridionally incompressible surface in $M_L\setminus K$, then it bounds a twisted $I$-bundle in $M_L\setminus K$.

\end{theorem}

\begin{corollary}
Let $K$ be a toroidally alternating knot in $S^3$ or $L(p,q)$ where $p$ is odd. Then $L(p,q)\setminus K$ contains no incompressible meridionally incompressible surfaces.\par

\end{corollary}

In addition, Ortel \hyperref[B17]{[17]} showed that the complement of a Montesinos knot or link $K$ contains a incompressible meridionally incompressible torus if and only if $K$ is a pretzel link $P(p,q,r,-1)$, where $\frac{1}{p}+\frac{1}{q}+\frac{1}{r}=1$. Ozawa \hyperref[B18]{[18]} showed that there is no closed incompressible meridionally incompressible surface in the complement of an algebraically alternating knot. We note that, the class of algebraically alternating links includes the class of Montesinos links defined by means of diagrams.

\section{Normal Position of Surfaces}\label{2}

Let $S$ be a closed orientable surface, $M=S\times I$ be a 3-manifold as the thickened surface, and $L \subset M$ be a link. We identify $S$ with $S\times \{ \frac{1}{2} \}$, then it has a projection $\pi : S\times I\to S$. Using the model in \hyperref[B16]{[16]}, we place a 3-ball at each crossing of $\pi(L)$, which we refer to as a \emph{bubble} $B$. See Figure \hyperref[the bubble]{1(a)}. At each crossing, both the overstrand and the understrand are in $\partial B$. We define $S_{+}$ to be the surface $S$ where the equatorial disk in each bubble is replaced by the upper hemisphere $\partial B_+$ of the bubble, and $M_{+}$ to be the 3-manifold bounded by $S_{+}$ and $S\times \{1\}$, so that $M_{+}$ does not intersect with the interior of any bubbles. Similarly, we can define $S_{-}$ and $M_{-}$ when replacing the equatorial disk in each bubble by a lower hemisphere $\partial B_-$ . \par

\begin{figure}[!htb]
\begin{subfigure}{.5\textwidth}
  \centering
  \includegraphics[width=.8\linewidth]{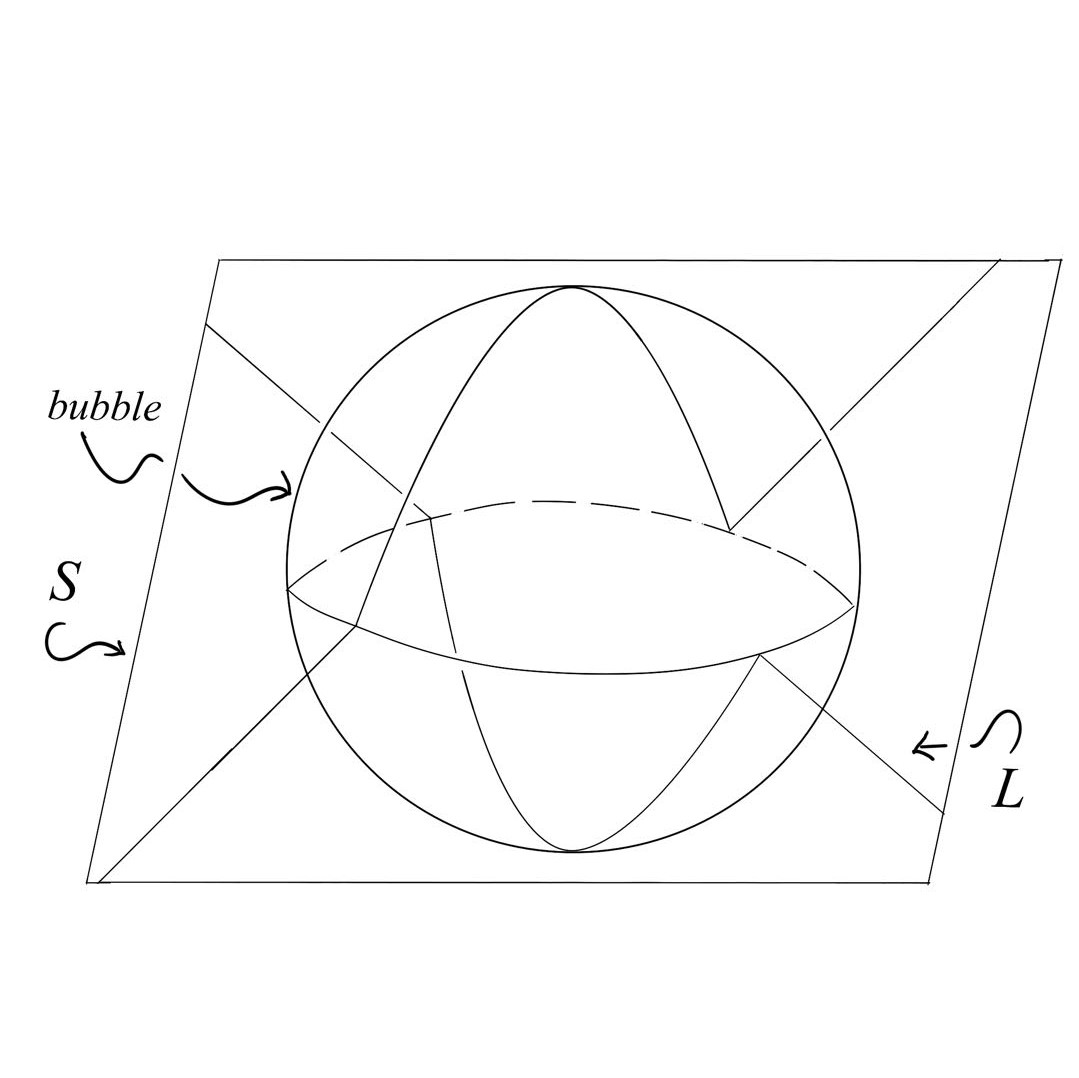}
  \caption{The bubble.}
  \label{the bubble}
\end{subfigure}
\begin{subfigure}{.5\textwidth}
  \centering
  \includegraphics[width=.8\linewidth]{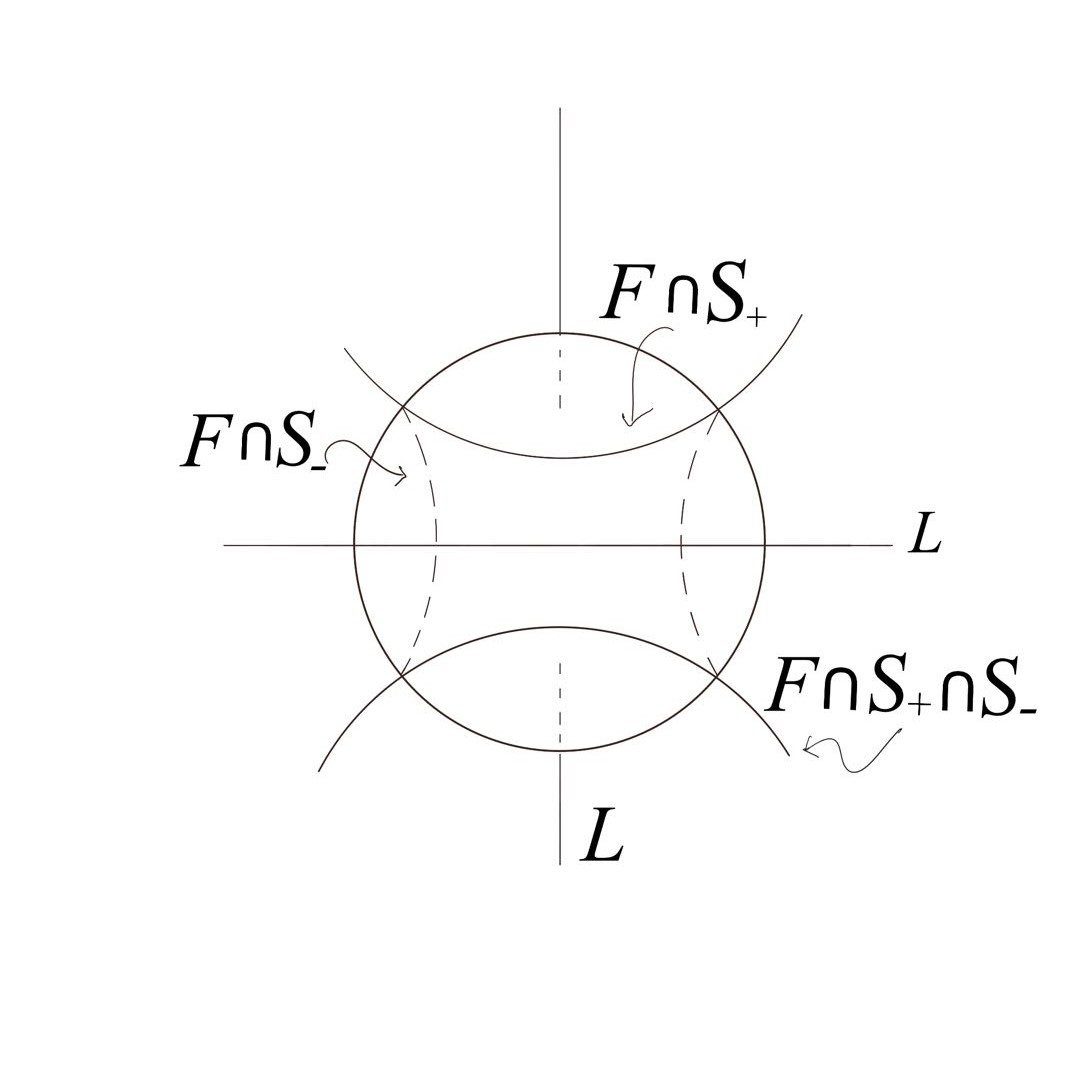}
  \caption{A local diagram of $F\cap S_{\pm}$.}
  \label{F0}
\end{subfigure}
\caption[The crossing ball.]{}
\label{Bubble}

\end{figure}

Let $F$ be an essential, meridionally incompressible surface embedded in $S\times I \setminus L$. We consider the following cases:\par

(1) $F$ is a sphere.\par
(2) There are no essential spheres in $S\times I \setminus L$, and $F$ is a closed, incompressible surface.

$F$ can be isotoped to intersect each bubble in a set of saddles. We associate an order pair $(s,i)$ to each embedding of $F$ prior to isotopy, in which $s$ is the number of saddles in $F$ and $i$ the number of intersection curves in $F\cap S_+$. For the rest of this paper, we assume the surface $F$ is connected and is chosen to minimize $(s,i)$ in lexicographical ordering. We say $F$ is in \emph{normal position} if the above assumptions are satisfied. We notice that, given a diagram of $F\cap S_{+}$, we can obtain its dual diagram of $F\cap S_{-}$ as shown in Figure \hyperref[F0]{1(b)}. \par

\section{Proofs}\label{3}

Denote each connected component of $F\cap M_{\pm}$ as $F_{i}^{\pm}$. Hereafter we refer to $S_+$ and $F\cap S_+$, but all arguments and constructions applies to both $F\cap S_+$ and $F\cap S_-$. Similarly, we refer to $M_+$ and each connected component $F_{i}^{+}\subset F\cap M_+$, but all arguments and constructions applies to $M_-$ and $F_{i}^{-}\subset F\cap M_-$. \par

\begin{proposition}\label{bopa}
Each connected component $F_{i}^{+}\subset F\cap M_+$ is incompressible, and  $\partial$-parallel to $S_+$ in $M_+$.\par

\end{proposition}

\begin{proof}
We first show that each connected component $F_{i}^{+}\subset F\cap M_{+}$ is incompressible in $M_{+}$. Suppose $F_{i}^{+}$ compresses in $M_{+}$. Then there exists a properly embedded compressing disk $D\subset M_+$ of $F_{i}^{+}$, with $\partial D\subset F_{i}^{+}$. Since $F\subset M$ is either a sphere or an incompressible surface,  $\partial D$ must bound a disk $D'$ in $F$, with $\partial D'=\partial D$. By the assumption that, in addition, there exists no essential spheres in $S\times I \setminus L$, we can isotope $D'$ in $S\times I \setminus L$ to $D$ while fixing its boundary. Suppose $D'\cap S_+=\emptyset$, then we could isotope $D$ to $D'$ in $M_+$ with $\partial D$ fixed in $F_{i}^{+}$. This would imply $D'\subset F_{i}^{+}$, contradicting the assumption that $D$ is a compressing disk of $F_{i}^{+}$. Therefore, $D'$ must intersect $S_+$ in a non-empty set of loop(s). Through a standard innermost loop argument, we can isotope $F$ to eliminate the loop(s) of $D'\cap S_+$. This means, we can reduce $i=|F\cap S_+|$ by at least one, contradicting the assumption $F$ is in normal position.\par

Now we have $M_{+}$ is a compression body, $\partial M_{+}= S_+ \cup  S \times \{1\}$, $(F_{i}^{+},\partial F_{i}^{+})\subset (M_{+},\partial M_{+})$, $F_{i}^{+}$ is incompressible, and $\partial F_{i}^{+} \cap S \times \{1\}=\emptyset$. Therefore, by a probably well-known result stated as Lemma 2.5 in \hyperref[B20]{[20]}, $F_{i}^{+}$ is $\partial$-parallel to $S_+$ in $M_+$.

\par

\end{proof}

We call each $F_{i}^{\pm}$ a \emph{dome}. As a result of the above proposition, it is reasonable to define $F_{i}^{+}$ (or $F_{i}^{-}$) as a \emph{lowest dome} if there are no component(s) of $F\cap M_{+}$ (resp., $F\cap M_{-}$) embedded in between the cobordism bounded by $F_{i}^{+}$ and $S_+$ (resp., $F_{i}^{-}$ and $S_-$). To compare this generalization with the original technique in \hyperref[B16]{[16]}, we notice if we take the thickened surface $S\times I$ as $S^2\times I$, and each simple closed curve of $F\cap S^2_{\pm}$ bounds a dome of disk . We will see our approach to proving Theorem \ref{mfa} can also be applied in \hyperref[B16]{[16]}. \par

When $F$ is in normal position, the following lemma guarantees the existence of nontrivial intersections $F\cap S_{+}$:\par

\begin{lemma}\label{fasp}
There exists an isotopy of $F$ such that the set of intersection curves $F\cap S_{+}$ is nonempty, and every intersection curve in $F\cap S_{+}$ intersects at least one bubble.\par

\end{lemma}

\begin{proof}
This is true because the proof of Lemma 5 (i), (ii) in \hyperref[B3]{[3]} applies.
\par

\end{proof}

Let $F_i$ denote a lowest dome of $F\cap M_+$, $M_i^+$ denote the cobordism bounded by $F_{i}^+$ and $S_+$, and $\partial B_+$ denote the upper hemisphere of a bubble $B$. Notice that as $F$ passes through a bubble $B$, the saddle corresponds to two intersection curves on $S_+$ that run parallel to the overstrand of $B$. $\partial B_+$ is divided by an overstrand of $L$ into two sides.  We proof the following technical lemma:\par

\begin{lemma}\label{rec}
Suppose $B$ is a bubble that intersects with a lowest dome $F_{i}^+$, then $M_i^+ \cap \partial B_+$ does not consists of any rectangle, whose boundary consists of two arcs of $F_{i}^+\cap \partial B_+$, and two arcs on the boundary of $\partial B_+$.

\end{lemma}

\begin{proof}

If $M_i^+ \cap \partial B_+$ consists of a rectangle $R$ shown in the below Figure \hyperref[f1]{2(a)}, which we call a sided rectangle, as $R$ does not contain the link strand on $\partial B_+$. Then we can pull a neighborhood of an arc on $F_{i}^+$ to the bubble to form a band connecting the pair of saddles intersecting $R$. And since $F_{i}^+$ is a lowest dome, there is no arc intersection of $F\cap S_+$ in the interior of $R$. So we can pull the two saddles and the band through the bubble and out the other side of $\partial B_+$, as shown in Figure \hyperref[f3]{3}. This would decrease the number of saddles, contradicting $F$ is in normal position.
 \par

If $M_i^+ \cap \partial B_+$ consists of a rectangle $R$ shown in Figure \hyperref[f2]{2(b)}, which we call a middle rectangle, as it contains the link strand on $\partial B$. This would contradict the assumption $F$ being meridionally incompressible. Because $F_{i}^+$ is a lowest dome, there is no arc intersection of $F\cap S_+$ in the interior of R. And we can find an arc $\mu$ on $F_{i}^+$ so that $\partial \mu$ is contained in a saddle $\sigma$ in $B$, which means $\sigma \cup \mu$ contains a meridional curve of $F$. \par

\end{proof}

\begin{figure}[!htb]
\begin{subfigure}{.5\textwidth}
  \centering
  \includegraphics[width=.8\linewidth]{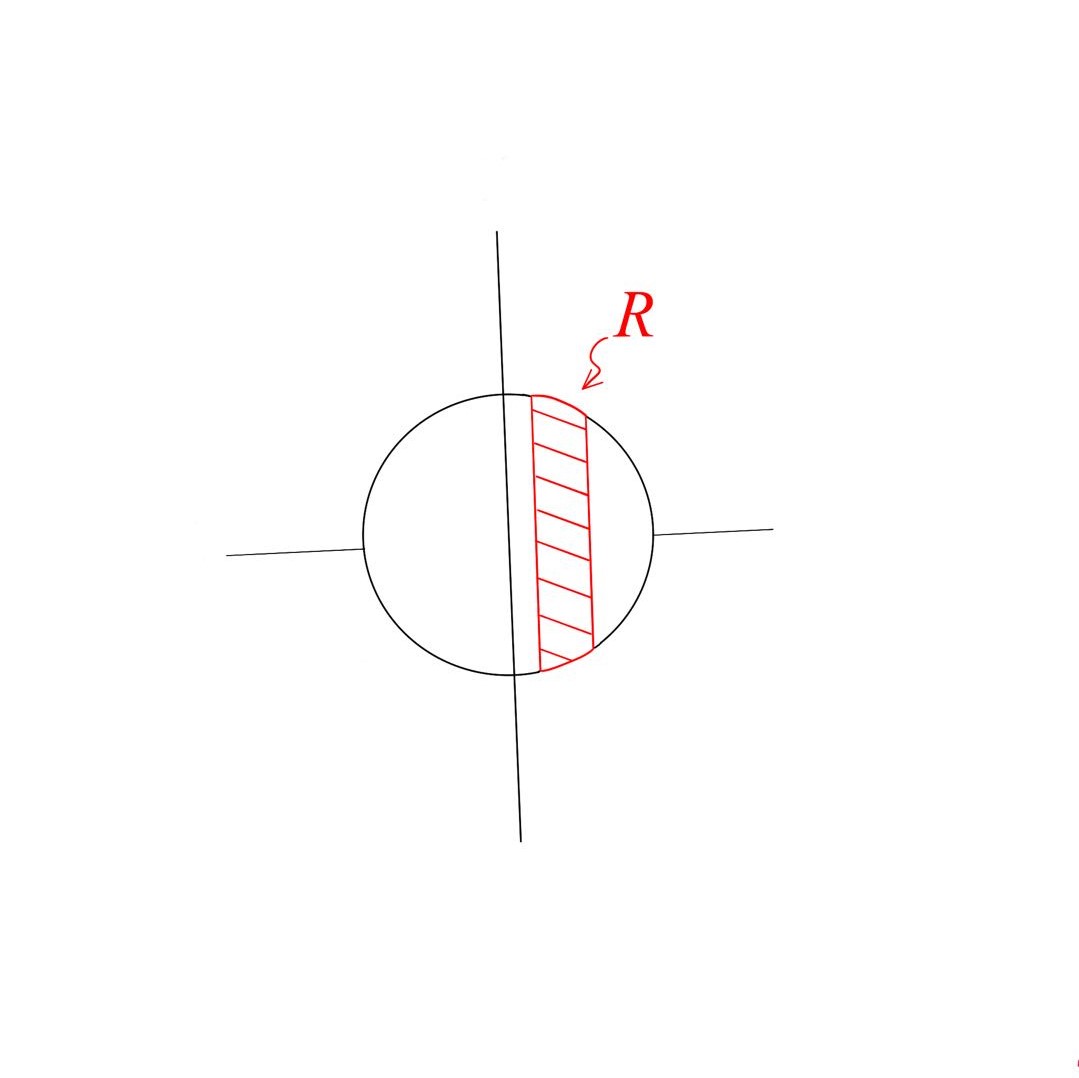}
  \caption{A sided rectangle.}
  \label{f1}
\end{subfigure}
\begin{subfigure}{.5\textwidth}
  \centering
  \includegraphics[width=.8\linewidth]{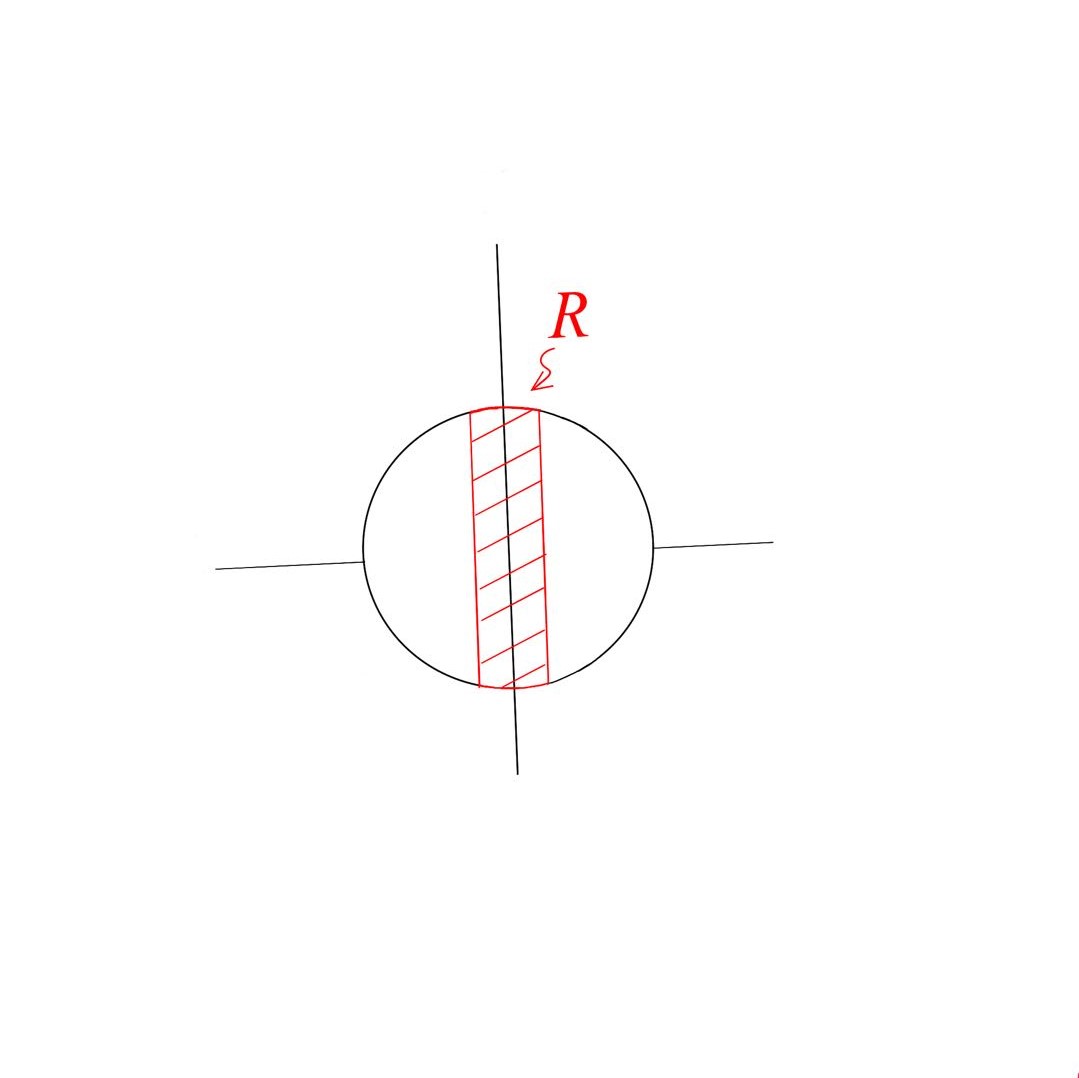}
  \caption{A middle rectangle.}
  \label{f2}
\end{subfigure}
\caption{The red shadowed parts represent subsets of $M_i^+\cap S_+$.}
\label{F}

\end{figure}

\begin{figure}[!htb] 
\centering
\includegraphics[width=0.25\textwidth]{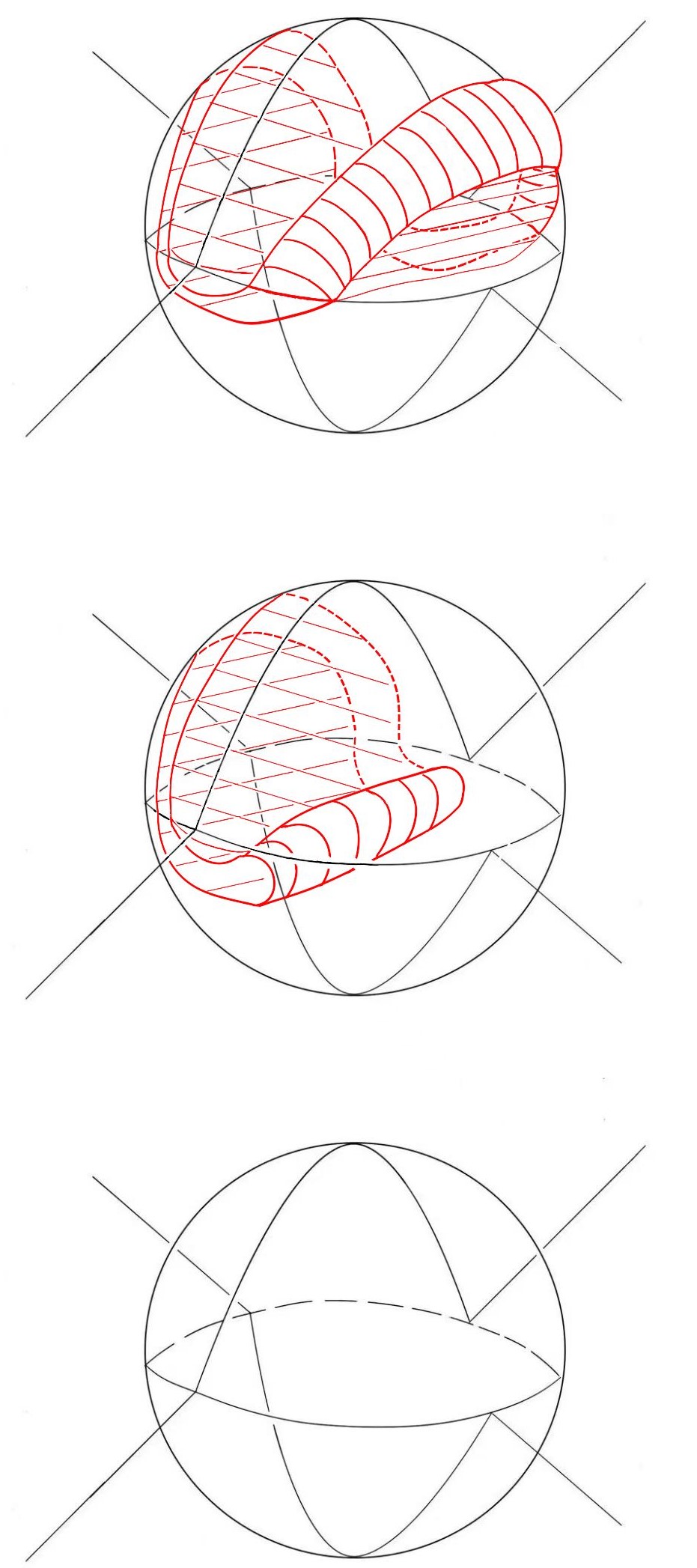}
\caption{Eliminate a sided rectangle.}
\centering

\label{f3}
\end{figure}

\subsection*{Proof of Theorem \ref{mfa}}

\begin{proof}
Assume by contradiction that the projection $\pi(L)$ of $L$ on $S$ is fully alternating, and $F$ is a closed essential meridionally incompressible surface embedded in $S\times I\setminus L$. We put $F$ in normal position. According to Lemma \ref{fasp}, $F\cap M_+$ is nonempty. And, by proposition \ref{bopa}, each connected component of $F\cap M_+$ is incompressible and $\partial$-parallel to $S_+$ in $M_+$. Hence, we can consider a lowest dome $F_{i}^+$ and the cobordism $M_i^+$ bounded by $F_{i}^+$ with $S_+$. By Lemma \ref{fasp}, every intersection curve in $F\cap S_{+}$ intersects at least one bubble. So we can assume there is a simple closed curve $C \subset F_{i}^+ \cap S_+$ that intersects with a bubble $B$.\par

We now claim that $M_i^+ \cap \partial B_+$ does not contain the overstrand on $\partial B_+$. Because suppose the overstrand of $\partial B_+$  is contained in $M_i^+ \cap \partial B_+$, by the assumption that $F_{i}^+$ is a lowest dome, $F_{i}^+$ would have to meet both sides of $\partial B_+$. Thus $M_i^+ \cap \partial B_{+}$ would consist of a middle rectangle, which would contradict Lemma \ref{rec}. But by alternating property, it follows that there must exist an arc $\alpha \subset C$ passes through one side of $\partial B_+$ such that the overstrand is on the right (similarly left), and then passes through the same side of $\partial B_+$
 such that the overstrand is on the left (similarly right), with $M_i^+ \cap \partial B_{+}$ consisting of a sided rectangle between the two passes. By Lemma \ref{rec}, this would contradict the assumption $F$ is in normal position. We note that the above proof also implies there does not exist any essential spheres in $S\times I\setminus L$.

\end{proof}

\section{Questions}
Given a certain compact 3-manifold, what is the class of knots or links in it, that a meridian lemma can apply to?  Is there a method to prove similar results without involving the projection surface, or a diagrammatic argument? \par

\section*{Acknowledgement}

I would like to thank William Menasco for his suggestions and many helpful discussions.


\begin{thebibliography}{99}

\bibitem{1}\label{B1}
{C. Adams, Generalized augmented alternating links and hyperbolic volumes, \emph{Algebr. Geom. Topol.} 17 (2017), no. 6, 3375–3397.}



\bibitem{2}\label{B2}
{C. Adams, Toroidally alternating knots and links, \emph{Topology} 33 (1994), no. 2, 353–369.}

\bibitem{3}\label{B3}
{C. Adams, C. Albors-Riera, B. Haddock, Z. Li, D. Nishida, B. Reinoso, and L. Wang, Hyperbolicity of links in thickened surfaces, \emph{Topology Appl.} 256 (2019), 262–278.}

\bibitem{4}\label{B4}
{C. Adams, J. Brock, John Bugbee,  T. Comar,  K. Faigin, A. Huston, A. Joseph, and D. Pesikoff, Almost alternating links, \emph{Topology Appl.} 46 (1992), no. 2, 151–165.}

\bibitem{5}\label{B5}
{C. Adams, A. Calderon, and N. Mayer, Generalized bipyramids and hyperbolic volumes of alternating $k$-uniform tiling links, \emph{Topology Appl.} 271 (2020), 107045, 28 pp.}


\bibitem{6}\label{B6}{F. Bonahon and L. Siebenmann, New geometric splittings of classical knots and the classification and symmetries of arborescent knots, (365 pages) available at http://www-bcf.usc.edu/~fbonahon/Research/Preprints/BonSieb.pdf}


\bibitem{7}\label{B7}{H. Fa,
Incompressible pairwise incompressible surfaces in almost alternating knot complements,
\emph{Topology Appl.} 80 (1997), no. 3, 239–249.}



\bibitem{8}\label{B8}
{J. Greene, Alternating links and definite surfaces, with an appendix by A. Juhasz, M Lackenby, \emph{Duke Math. J.} 166 (2017), no. 11, 2133-2151.}


\bibitem{9}\label{B9}{J. Hass, A. Thompson and A. Tsvietkova,
The number of surfaces of fixed genus in an alternating link complement,
\emph{Int. Math. Res. Not.}, IMRN 2017, no. 6, 1611–1622.}



\bibitem{10}\label{B10}
{J. Howie, A characterisation of alternating knot exteriors, \emph{Geom. Topol.} 21 (2017), no. 4, 2353-2371.}



\bibitem{11}\label{B11}
{J. Howie, Surface-alternating knots and links, Ph.D. thesis, University of Melbourne, 2015.}


\bibitem{12}\label{B12}
{J. Howie, J. Purcell, Geometry of alternating links on surfaces, \emph{Trans. Amer. Math. Soc.} 373 (2020), no. 4, 2349–2397.}


\bibitem{13}\label{B13}
{W. Lickorish, M. Thistlethwaite, Some links with nontrivial polynomials and their crossing‐numbers, \emph{Comment. Math. Helv}. 63 (1988), no. 4, 527‐539.}

\bibitem{14}\label{B14}{M. Lozanoand, J. Przytycki, Incompressible surfaces in the exterior of a closed 3-braid I, surfaces with horizontal boundary components, \emph{Math. Proc. Camb. Phil. Soc.} 98 (1985), 275–299.}


\bibitem{15}\label{B15}
{E. Mayland and K. Murasugi, On a structural property of the groups of alternating links, \emph{Canad. J. Math.} 28 (1976), no. 3, 568‐588.}

\bibitem{16}\label{B16}
{W. Menasco, Closed incompressible surfaces in alternating knot and link complements. \emph{Topology} 33, no.2(1984): 353-369.}



\bibitem{17}\label{B17}{U. Oertel, Closed incompressible surfaces in complements of star links, \emph{Pacific J. Math.}, 111 no. 1 (1984), pp. 209---230.}

\bibitem{18}\label{B18}
{M. Ozawa, Rational structure on algebraic tangles and closed incompressible
surfaces in the complements of algebraically alternating knots and links, \emph{Topology Appl.} 157 (2010), no. 12, 1937–1948.}

\bibitem{19}\label{B19}
{P. Ozsváth and Z. Szabó, On the Heegaard Floer homology of branched double covers, \emph{Adv. Math.} 194 (1) (2005), 1‐33. }


\bibitem{20}\label{B20}
{M. Scharlemann, Proximity in the curve complex: boundary reduction and bicompressible surfaces. \emph{Pacific J. Math.} 228 (2006), no. 2, 325–348.}


\bibitem{21}\label{B21}
{W. Thurston, On the Geometry of Topology of 3-manifolds. \emph{Princeton Notes}.}








\end{thebibliography}
\end{document}